\documentclass[a4paper,12pt]{amsart}

\usepackage{amsmath}
\usepackage{amssymb}
\usepackage{amsthm}
\usepackage{mathrsfs}
\usepackage{graphicx}
\usepackage{enumerate}

\newtheorem{theorem}{Theorem}[section]
\newtheorem{lemma}[theorem]{Lemma}
\newtheorem{proposition}[theorem]{Proposition}
\newtheorem{corollary}[theorem]{Corollary}
\newtheorem{definition}[theorem]{Definition}
\newcommand{\proba}{\mathbb{P}}

\font\labbfont=bbold10 scaled \magstep2
\newcommand{\gun}{\hbox{\labbfont \char49}}
\newcommand{\ambigus}{\mathscr{Q}}
\newcommand{\ftm}{{full topological Markov expanding map of the interval }}
\newcommand{\code}{W}

\newcommand{\N}{\ensuremath{\mathbb{N}}}    
\newcommand{\R}{\ensuremath{\mathbb{R}}}     
\newcommand{\Z}{\ensuremath{\mathbb{Z}}}    
\renewcommand{\P}{\ensuremath{\mathbb{P}}}    
   
\newcommand{\shift}{\ensuremath{\mathscr{S}}}

\newcommand{\cardin}{\mathcal{K}}
\newcommand{\parti}{\mathcal{P}}
\numberwithin{equation}{section}

\begin{document}
\title{Chains of infinite order, chains with memory of variable length, 
 and maps of the interval.}
\author{Pierre Collet }
\address{
  Centre de Physique Th\'eorique, CNRS UMR 7644, Ecole Polytechnique,
  91128 Palaiseau Cedex, France} 
 \email{collet@cpht.polytechnique.fr}

\author{Antonio Galves} 
\address{
  Instituto de Matem\'atica e Estat\'{\i}stica, Universidade de S\~ao
  Paulo, BP 66281, 05315-970 S\~ao Paulo, Brasil} 
  \email{galves@usp.br}

\date{July 29, 2012}

\maketitle

\begin{abstract}
We show how to construct a topological Markov map of
the interval whose invariant probability measure is the stationary law
of a given stochastic chain of infinite order. In particular we
caracterize the maps corresponding to stochastic chains with memory of
variable length. The problem treated here is the converse of the
classical construction of the Gibbs formalism for Markov expanding
maps of the interval.  
\end{abstract}

\section{Introduction.}\label{introduction}

The founding papers by Bowen (2008)\nocite{bowen}, Ruelle (1978)\nocite{ruelle} and Sina{\u\i} (1972)\nocite{sinai} explained how to use the Gibbs formalism for Markov expanding maps of the interval. In this formalism to each such map of the interval is associated a Gibbs measure which corresponds through the dynamical coding to an absolutely continuous invariant measure.  Recalling that Gibbs measures with H\"older continuous interactions are stochastic chains of infinite order (cf. Fern\'andez and Mailard 2004 \nocite {fernandez_maillard_2004} and references therein), this means that expanding maps of the interval are naturally associated to stochastic chains.  In particular, piecewise affine topological Markov maps correspond to Markov chains on a finite alphabet.

In this paper we address the converse problem, namely, given a
stochastic chain of infinite order, taking values on a finite
alphabet, can we construct a topological Markov map of the interval
whose invariant measure is the invariant probability measure of the
chain?  

A particular case of this question has to do with the class of
stochastic chains with memory of variable length, introduced by
Rissanen (1983)\nocite{rissanen}.  Recently C\'enac et
al. \cite{cenac} have shown how to represent two interesting examples
of stochastic chains with memory of variable (unbounded) length by
maps of the interval.  Inspired by this paper, we discuss at a more
general level some of the relations between chains of infinite order,
chains with memory of variable length models and expanding Markov maps
of the interval.

This paper is organized as follows. In Section \ref{recall} we briefly present the notions of expanding maps of the interval and stochastic chains of infinite order and for the convenience of the reader we recall some classical results. In Section \ref{map_to_chain} we recall the classical construction of a stochastic chain of infinite order given an expanding map of the interval. For more details about this construction we refer the reader to the articles of \cite{walters1978a}, \cite{walters1978b} and \cite{hofkel} and references therein. In Section \ref{chain_to_map} we explain how to construct an expanding map of the interval given a stochastic chain of infinite order. Finally in Section \ref{vl} we study the particular case of stochastic chains with memory of variable (unbounded) length.

\section{Notation, chains and maps.}\label{recall}

In order to make this paper self contained as much as possible, we
gather in this section some basic definitions and results about
stochastic chains and maps of the interval. 

Let $A$ denote a finite alphabet $A=\{1,\ldots, \mathcal{K}\}$.  Given two
integers $m\leq n$ we denote by $w_m^n$ the sequence $(w_m, \ldots,
w_n)$ of symbols in $A$, and $A_m^n$ denotes the set of such
sequences. Any sequence $w_m^n$ with $m > n$ represents the empty
string. The same notation is extended to the cases $m=\pm \, \infty$.

Given two finite sequences $w$ and $v$ we will denote by $vw$ the
sequence obtained by concatenating the two strings.  For example,
$z_{-\infty}^{-1}a$ denotes the sequence having the symbol $a$ at the
zero position and the symbols $z_i$ at the positions $i\le -1$.

For a finite string $a_{m}^{n} \in A_m^n$, we denote by
$C(a_{m}^{n})$ the cylinder given by
$$
C(a_{m}^{n})=\big\{x_{-\infty}^{+\infty}\in A_{-\infty}^{+\infty}\;:\; x_{m}^{n}=a_{m}^{n}\big\}\;.
$$

\subsection{Stochastic chains of infinite order.}
A family $p$ of  numbers $p(a|x_{-\infty}^{-1})\in[0,1]$, with $a\in A$ and
$x_{-\infty}^{-1}\in A_{-\infty}^{-1}$, is called a family of
transition probabilities if it satisfies the two conditions
\begin{itemize} 
\item For each fixed sequence $x_{-\infty}^{-1}$
$$
\sum_{a}p(a|x_{-\infty}^{-1})=1\;.
$$
\item For each symbol $a\in A$, the map 
$$
x_{-\infty}^{-1}\longrightarrow p(a|x_{-\infty}^{-1})
$$
is measurable with the product sigma-algebra on  $A_{-\infty}^{-1}$.
\end{itemize}

\begin{definition}\label{nonnull}
A family $p$ of
transition probabilities satisfies the condition of non-nullness if
 \[
 \inf\{ p(a|x_{-\infty}^{-1})\;\colon\; a\in A,
   x_{-\infty}^{-1}\in A_{-\infty}^{-1}\}  > 0\;.
\]
 \end{definition}

\begin{definition}\label{betak}
The continuity rate of a  family $p$ of
transition probabilities is the sequence $(\beta_{k})_{k\ge1}$ defined by
\[
      \beta_k \, = \,\sup \Bigl\{ \,| 
  p(a | x_{-\infty}^{-1})-p(a | y_{-\infty}^{-1})
|\;\colon\; a\in A, x_{-\infty}^{-1}, y_{-\infty}^{-1} 
\in A_{-\infty}^{-1}\mbox{ with } \, x_{-k}^{-1}=y_{-k}^{-1}
      x_{-\infty}^{-1} \overset{k}{=} y_{-\infty}^{-1}\,\Bigr\}.
\]
\end{definition}

\begin{definition}\label{continuity}
The family $p$ of
transition probabilities with continuity rate $(\beta_{k})$ is said to
be continuous if
\[
\lim_{k\rightarrow +\infty}\beta_k=0\, .
\]
\end{definition}

\begin{definition}
 We will say that a probability measure $\proba$ on $A^{\Z}$ is
 translation invariant (or stationary) if for any $m\ge 0$ and for any
 $a_{0}^{m}\in A_{0}^{m}$, we have
 \[
\proba\big\{X_n^{n+m}=a_{0}^{m}\big\}=\proba\big\{X_0^{m}=a_{0}^{m}\big\}
\]
for any $n \in \Z$.
\end{definition}
The notion of translation invariance says that the probability measure
$\P$ is invariant with respect to the shift $\shift: A^{\Z}
\rightarrow A^{\Z}$, defined as follows. For every sequence
$\underline{z}=z_{-\infty}^{+\infty}$, we have
\[
\shift(\underline{z})_i=\underline{z}_{i-1}\, .
\]

v\begin{definition}
 We will say that a probability measure $\proba$ on $A^{\Z}$ is invariant
 with respect to  $p$, if for any continuous function
 $f:A_{-\infty}^0\rightarrow \R$ we have
\begin{equation}\label{invariant}
\int f(z_{-\infty}^0) d\P(z_{-\infty}^0)= 
\int
\sum_{a \in A} p(a\,|\,(z_{-\infty}^{-1}) f(z_{-\infty}^{-1}a) d\P(z_{-\infty}^{-1})
\end{equation}
\end{definition}

From stationarity and invariance of $\P$ with respect to $p$ it
follows immediately that for any pair $m\le n$ of integers and for any
$a_{m}^{n}\in A_{m}^{n}$, we have
$$
\proba\big(C(a_{m}^{n})\big)=\int_{A_{-\infty}^{m-1}}
\prod_{j=m}^{n}\;p\big(a_{j}\big|a_{m}^{j-1}x_{-\infty}^{m-1}\big)\;
d\proba\big(x_{-\infty}^{m-1}\big) \;.
$$

For later references, it is convenient to collect in the following
theorem some well know results about families of transition
probabilities. 
\begin{theorem}\label{existinv}
  If the family of transition probabilities satisfies the non-nullness condition \ref{nonnull} and the sequence of continuity rates is summable, then there exists a unique ergodic stationary probability measure $\proba$ on $A^{\Z}$,
 invariant with respect
to  $p$. This invariant probability measure 
has no atom, and for any finite sequence $a_{m}^{n}$,
$\proba(C(a_{m}^{n}))>0$. 
\end{theorem}

This type of result has been proved by many authors starting with Onicescu and Mihoc (1935)\nocite{onimih35}, 
Doeblin and Fortet (1937)\nocite{doefor37} and Harris (1955)\nocite{har55} and 
Comets {\sl et al.} (2002) \nocite{comets2002}.

Let us now consider the probability space having $A^{\Z}$ as sample
space, equipped with its product $\sigma$-algebra, and having
$\proba$, whose existence is granted by Theorem \ref{existinv}, as
probability measure. We can define a stochastic chain $(X_n)_{n\in
  \Z}$ on this probability space, by taking, for each $n \in \Z$
$$
X_n: A^{\Z} \longrightarrow A
$$ as the projection on the $n^{th}$ coordinate.  In other words, for
any $m \le n$ and any choice of the sequence $a_m^n$, we have
 $$
\proba\big(C(a_{m}^{n})\big)=\P\big\{X_{m}^{n}=a_{m}^{n}\big\}\;.
$$ The stochastic chain $(X_n)_{n\in\Z}$ is said to be associated to
the family of transition probabilities $p$.
 \bigskip

\subsection{Piecewise expanding maps of the interval.}

From now on let $\Omega=[0,1]$. We first recall the definition of a
piecewise expanding map of the interval $\Omega$.  Let
$0=\eta_{0}<\eta_{1}<\ldots<\eta_{\cardin}=1$ be a finite sequence and
for each interval $I_j=]\eta_{j-1},\eta_{ j}[$, with ($1\le j\le
    \cardin$), let $T_{j}$ be a monotone map from $I_j$ to $\Omega$
    which extends to a $C^{2}$ map on
    $\bar{I_j}=[\eta_{j-1},\eta_{j}]$. The map $T$ is defined as
    follows. For each $\omega \in \Omega\backslash \{\eta_0,
    \eta_1,\ldots, \eta_{\cardin}\}$
$$
T(\omega)=T_{j}(\omega)\, , \quad\mathrm{if}\quad \omega\in I_j\;.
$$ We denote by $\parti$, the collection of open intervals $I_j$, with
$j=1,\ldots,\cardin$, and observe that it defines a partition of the
$\Omega\backslash\mathscr{N}_0$, where $\mathscr{N}_0=\{\eta_0,
\eta_1,\ldots, \eta_{\cardin}\}$. From now on let us call
$A=\{1,\ldots,\cardin\}$ the set of indexes of the partition $\parti$.

\begin{definition}
The map $T$ of the interval has the (uniform) 
\index{expanding property}expanding property
if there is an integer $m>0$ and a constant $c>1$ such that
at any point where  $T^{m}$ is differentiable we have
$$
\big|{T^{m}}'\big|\ge c\;.
$$ 
\end{definition}

\begin{definition}
The piecewise expanding map $T$ of the interval is said to be
topological Markov if for any $i=1,\ldots,\cardin$, the closure of
$T(I_{i})$ is a union of closures of intervals $I_{j}$,
$j\in\{1,\ldots,\cardin\}$. The map $T$ is called full topologicall
Markov if for any $T_i(I_i)=\Omega$ for any $i=1,\ldots,\cardin$.
\end{definition}
Note that the topological Markov property notion is not to be mistaken
with the Markov property of stochastic processes.

Recall that the map $T$ is not defined on the finite set
$\mathscr{N}_0$. Call $\mathscr{N}$ the set of pre-images of
$\mathscr{N}_0$, namely 

\[
\mathscr{N}=\mathscr{N}_0\cup \bigcup_{k\ge 1} \left\{\omega \in
\Omega\, \big| \, \, 
T^k(\omega) \in \mathscr{N}_0\right\}
\]

Given an \ftm  $T$,  we define a
coding of $\Omega\backslash \mathscr{N}$ with alphabet
$A=\{1,\ldots,\cardin\}$.  
This coding is  a map  $\code$ from $\Omega\backslash \mathscr{N}$ to
$A^{\N}$  
$$
\omega\longrightarrow \code(\omega)=\big( \code_{n}(\omega)\big)_{n \in \N}
$$ 
given  by
$$
\code_{n}(\omega)=j\, , \quad\mathrm{if}\quad T^{n}(\omega)\in I_j\;.
$$

Given a \ftm $T$, we have just associated a code to a point in $\Omega\backslash
\mathscr{N}$.  We can also go in the opposite direction and this is
the content of the next proposition.
To simplify the presentation we will restrict
ourselves to the case of full Markov maps. The extension to the case
of general Markov maps is straightforward.  

\begin{proposition}
  Assume $T$ is a \ftm and $A$ is the set of indexes of the partition
  $\parti$. Then given a code $x_{0}^{+\infty} \in A_{0}^{+\infty}$,
  there exists at most one point in the interval $\Omega$ which is
  coded by this sequence.
\end{proposition}

Both directions are well known, see for instance \cite{walters1978a}
and \cite{walters1978b}.

\section{Constructing a chain from a map}\label{map_to_chain}

Let $\mu$ be a $T$-invariant measure defined on $\Omega=[0, 1]$.  We
now have the three ingredients of a probability space, namely the
sample space $\Omega=[0,1]$, with its Borel $\sigma$-algebra and the
$T$-invariant probability $\mu$.  Furthermore, the coding associated
to the map $T$ defines a sequence of random variables $(\code_{n})_{n
  \in \N}$ with values in the alphabet $A$.

Let us denote by $q(x_{m}^{n})$ (with $m\le n$ belong to $\Z$)
the cylinder probabilities on
$A^{\Z}$ defined by
$$
q\big(x_{m}^{n}\big)=
\mu\big\{\omega\in\Omega\,,
\,\code_{0}(\omega)=x_{n},\, \code_{1}(\omega)=x_{n-1},\,\ldots,
\code_{n-m}(\omega)=x_{m}
\big\}\;.
$$
The  time was reversed in the definition of $q$ to follow the usual
convention for  stochastic processes.

Kolmogorov's Existence Theorem implies that there exists a unique stationary
probability measure $\proba$ on $A^{\Z}$ such that for any integers
$m\le n$, and any sequence $x_{m}^{n}\in A^{n}_{m}$ we have
$$
\proba\big\{ X_{m}^{n}=x_{m}^{n}\big\}=q\big(x_{m}^{n}\big)\;,
$$
where $X_{n}\,:\, A^{\Z}\rightarrow A$  is the projection on the
$n^{\mathrm{th}}$ coordinate. The $(X_{n})_{n\in \Z}$ is in general a
chain of infinite order. The next theorem will give an explicit
expression for its family of transition probabilities
$$
p\big(b\,\big|\,a_{-\infty}^{-1}\big)=
\proba\big(X_{0}=b\,\big|\,X_{-\infty}^{-1}=a_{-\infty}^{-1}\big)\;.
$$

We denote by $\lambda$ the Lebesgue measure of the interval $\Omega=[0,1]$.

\begin{theorem}\label{maptochain}
Let $\Omega=[0,1]$ and let $T$ be a \ftm. Assume that the Lebesgue
measure $\lambda$ is invariant and ergodic with respect to $T$. Then
the family of transition probabilities of the associated chain of
infinite order is given by
$$
p\big(b\,\big|\,a_{-\infty}^{-1}\big)=\lim_{n\to\infty}
\frac{1}{\big|T'(\omega_{n})\big|}
$$
where $\omega_{n} $ is any point in $\Omega$, such that 
$$
\code_{0}^{n}(\omega_{n})=\big(b,a_{-1},\ldots,a_{-n}\big)\;.
$$
\end{theorem}
For the proof of Theorem 3.1 we ferer the reader to \cite{walters1978a}, \cite{walters1978b} and \cite{ledrappier1974}.

In the general case, where the invariant abosutely continuous
invariant measure $\mu$ is not the Lebesgue measure $\lambda$, we have
the folowing result.

\begin{corollary}
Let $T$ be a topological Markov piecewise expanding map
on $\Omega$. Assume that the probability measure $\mu$ which is
absolutely continuous with respect to the Lebesgue measure $\lambda$ 
is invariant and ergodic with respect to $T$.
 Then the
family of transition probabilities of the associated chain of
infinite order is given by 
$$
p\big(b\,\big|\,a_{-\infty}^{-1}\big)
=\lim_{n\to\infty}\frac{g(\omega_{n})}{g(T(\omega_{n}))\big|T'(\omega_{n})\big|}
$$
where $g=d\mu/d\lambda$, and  $\omega_{n}$ is any point such that 
$$
\code_{0}^{n}(\omega_{n})=\big(b,a_{-1},\ldots,a_{-n}\big)\;.
$$
\end{corollary}

\begin{proof}
Let $G$ be the distribution of $\mu$ defined in the usual way by
$$
G(t)=\mu\big([0,t])\;.
$$
Obviously $G$ is a non decreasing function which is also continuous
since $\mu$ is absolutely continuous with respect to the Lebesgue
measure $\lambda$. By a theorem of Buzzi (1997)\nocite{buzzi} (see also
Liverani (1995)\nocite{liverani}), $\mu$ is equivalent to
the $\lambda$, and $g=d\mu/d\lambda=G'$ is a continuous non-vanishing
function. In other words  $G$ is a $C^{1}$ diffeomorphism. 

Consider $G^{-1}$ as a random variable defined on
the probability space $(\Omega,\parti_{\infty}, \lambda)$. This fact
together with the invertibility of $G$ implies that the Lebesgue
measure $\lambda$ is invariant and ergodic with respect to the map 
$T_{0}$ defined  by 
$$
T_{0}=G\circ T\circ G^{-1}\;.
$$
Theorem \ref{maptochain} applies to $T_{0}$, and the corollary follows
by the chain rule. 
\end{proof}

\section{Constructing a map from a chain.}\label{chain_to_map}

Let $(X_{n})$ be a stationary ergodic stochastic chain taking values
in the finite alphabet $A=\{1,\ldots, \cardin\}$, and defined on a
probability space $\big(A^{\Z},\mathscr{F},\proba\big)$. Let us assume
that the law $\proba $ of the chain has no atom. Our goal is to define a map  $T: \Omega \rightarrow \Omega$, where $\Omega =[0,1]$, such that the construction of Section \ref{map_to_chain} recovers the chain $(X_{n})$.
The map $T$ will be defined by a conjugation to the shift $\shift$ through a map $h:A_{-\infty}^0 \rightarrow \Omega$ defined below.

We define a  distance on $A_{-\infty}^{0}$ as follows.
\begin{definition}\label{dist}
First of all, for two sequences $x_{-\infty}^{0}$ and $y_{-\infty}^{0}$, denote by
$\delta\big(x_{-\infty}^{0}, y_{-\infty}^{0}\big)$  
 the nearest position to the origin
where these two sequences differ, namely
$$
\delta\big(x_{-\infty}^{0}, y_{-\infty}^{0}\big)=\min\left\{n \ge 0 :\,x_{-n}\not=
y_{-n}\right\}\;. 
$$ 
For a fixed number $0<\zeta<1$, we define the distance $d$ on $A_{-\infty}^{0}$
by
\[
d\big(x_{-\infty}^{0}, y_{-\infty}^{0}\big)=\zeta^{\delta\big(x_{-\infty}^{0}, y_{-\infty}^{0}\big)}\;.
\]
\end{definition}

We denote by $<$ the lexicographic order on $A_{-\infty}^0$. 
Namely $x_{-\infty}^{0} <y_{-\infty}^{0}$, if for some $m\ge 0$, we
have $x_{-(m-1)}^0=y_{-(m-1)}^0$ and $x_{-m}<y_{-m}$. 
For a point $x_{-\infty}^0\in A_{-\infty}^0$, we denote by
$J(x_{-\infty}^{0 })$ the set of points 
   $$
J(x_{-\infty}^{0 })=\big\{y_{-\infty}^{0}\;\big |\; y_{-\infty}^{0}
\le  x_{-\infty}^{0} \big\}\;. 
$$ 
We define the map $h$ from $A_{-\infty}^0$ to $\Omega$ by 
$$
h\big(x_{-\infty}^{0}\big)=\proba\big(J(x_{-\infty}^{0})\big)\;.
$$ 
 
Before stating the properties of the map $h$, we need to define a
countable set $\mathscr{Q}$ of exceptional codes, given by
$$
\mathscr{Q}=\bigcup_{j=1}^{\cardin-1}
\big\{\cardin_{-\infty}^{-1}j,1_{-\infty}^{-1}(j+1)\big\}  
\bigcup_{k=0}^{\infty}\bigcup_{x_{-k}^{0}\in A_{-k}^{0}}
\bigcup_{j=1}^{\cardin-1}\big\{\cardin_{-\infty}^{-1}jx_{-k}^{0},
1_{-\infty}^{-1}(j+1)x_{-k}^{0}\big\}\;,
$$
where $\cardin_{-\infty}^{-1}$ and $1_{-\infty}^{-1}$ denote the
sequences identically equal to $\cardin$ and $1$, respectively.

\begin{proposition}\label{hcontinu}
Let $p$ be a family of transition probabilities satisfying the  non-nullness condition \ref{nonnull}. 
The map $h$ defined above has the following properties

\begin{enumerate}[i)]
\item $h$ is non decreasing on $\Omega$ and strictly increasing
  outside $\ambigus$;
\item  $h$ is continuous;
\item  $h$ is invertible except on the
countable set $h(\ambigus)$, and the set of preimages of any point in
$h(\ambigus)$ has cardinality at most two;
\item the inverse function $h^{-1}$ is continuous outside $h(\ambigus)$;
\item  the image of $\proba$ by $h$ is the Lebegue measure on $\Omega$;
\item finally $h$ is sujective.
\end{enumerate} 
\end{proposition}

\begin{proof}
We first prove that 
 the map $h$ is injective except on the countable set $\ambigus$.
Let  $x_{-\infty}^{0}<y_{-\infty}^{0}$. This means
that $x_{0}<y_{0}$, or there exists an integer $k\ge0$ such that
$x_{-k}^{0}=y_{-k}^{0}$, and  $x_{-(k+1)}<y_{-(k+1)}$. Assume
$x_{-\infty}^{0}\notin \ambigus$. This implies that for infinitely
many indices $n$, we have $x_{-n}\le \cardin-1$. Let $m>k$ be such an
index. For any $z_{-\infty}^{0}$ in the cylinder $C(\cardin
x_{-(m-1)}^{0})$ we  have
$$
x_{-m}^{0}<z_{-m}^{0}<y_{-m}^{0}\;.
$$
Therefore 
$$
J(x_{-m}^{0})\cap C(\cardin x_{-(m-1)}^{0})=\emptyset\;,\qquad\mathrm{and}\qquad
C(\cardin x_{-(m-1)}^{0})\subset J(y_{-m}^{0})\;.
$$
From Theorem \ref{existinv} we have $\proba(C(\cardin x_{-(m-1)}^{0}))>0$,
hence 
$$
h\big(x_{-m}^{0}\big)=\proba\big(J(x_{-m}^{0})\big)<
\proba\big(J(y_{-m}^{0})\big)=h\big(y_{-m}^{0}\big)\;.
$$
The case where $y_{-\infty}^{0}\notin \ambigus$ can be treated
similarly. 

If $x_{-\infty}^{0} \in \ambigus$ and $y_{-\infty}^{0}\in \ambigus$
but
$$
\big(x_{-\infty}^{0},y_{-\infty}^{0}\big)\neq
\big(\cardin_{-\infty}^{-1}x_{0},1_{-\infty}^{-1}(x_{0}+1\big) 
$$
and for any $k\ge 0$, 
$$
y_{-\infty}^{0}\neq 1_{-\infty}^{-k-1}(x_{-k-1}+1)x_{-k}^{0}
$$
then there exists $\tilde x_{-\infty}^{0}\notin \ambigus$ and such
that $x_{-\infty}^{0}<\tilde x_{-\infty}^{0}< y_{-\infty}^{0}$. From
above it follows that
$$
h\big(x_{-\infty}^{0}\big)<h\big(\tilde x_{-\infty}^{0}\big)
< h\big(y_{-\infty}^{0}\big)\;.
$$

Finally, if for some $a\in \{1,\ldots, \cardin-1\}$ we have either
\[
x_{-\infty}^0=\cardin_{-\infty}^{-1)}a\qquad \mbox{and}\qquad  y_{-\infty}^0=1_{-\infty}^{-1}(a+1)\, ,
\] 
or 
\[
x_{-\infty}^0=\cardin_{-\infty}^{-(k+2)}ax_{-k}^0 \qquad \mbox{and}
\qquad y_{-\infty}^0=1_{-\infty}^{-(k+2)}(a+1)x_{-k}^0\, ,
\]
for some $k\ge 1$, then $h(x_{-\infty}^0)=h(y_{-\infty}^0)$. This concudes the proof of (i).

Now let us prove that the map $h$ is continuous. Take 
$x_{-\infty}^{0}\in A_{-\infty}^{0}$, and  let $(y_{-\infty}^{0}(n))$
be a sequence in $A_{-\infty}^{0}$ converging to $x_{-\infty}^{0}$ in
the metric defined in \ref{dist}. This implies that for any
$k$ there exists $\bar n(k)$ such that for any $n\ge\bar n(k) $,
$y_{-k}^{0}(n)= x_{-k}^{0}$. This implies
$$
J(y_{-\infty}^{0}(n))\Delta J(x_{-\infty}^{0})\subset C(x_{-k}^{0})\;,
$$
and therefore
$$
\big|h(y_{-\infty}^{0}(n))-h(x_{-\infty}^{0})\big|\le \proba\big(C(x_{-k}^{0})\big)\, .
$$
By Theorem \ref{existinv} the probability measure $\proba$ has no atoms, hence
$\proba\big(C(x_{-k}^{0})\big)$ tends to $0$, when $k$ tends to $\infty$, proving that $h$ is
continuous. This concludes the proof of (ii).

Assertion (iii) and (iv) follow immediately from (i) and (ii). 

Finally to prove (v), take $z\in \Omega\backslash h(\ambigus)$. The
inverse value $h^{-1}(z)$ is uniquely defined, and therefore 
$$
\lambda\big([0,z]\big)=z=h\big(h^{-1}(z)\big)=
\proba\big(J(h^{-1}(z))\big)=\proba\big(h^{-1}([0,z])\big)\;. 
$$ Since the measure $\proba$ and the Lebesgue measure have no atoms,
the same result holds for the countable set of points in
$h(\ambigus)$.  This implies by standard measure theoretic arguments
(see for example Breiman 1992\nocite{breiman}) that $\lambda$ is the
image of $\proba)$ by $h$. This concludes the proof of (v).

Finally (vi) follows from the fact that the measure $\proba$ has no
atom by Theorem \ref{existinv} the map $h$ is continuous and hence
surjective.
\end{proof}

We define the map $T$ on $\Omega\backslash\ambigus$ by
$$
T=h\circ \shift \circ h^{-1}\;.
$$
More explicitly, for  $z\in \Omega\backslash \ambigus$ we have
\begin{equation}\label{Tegal}
T(z) =\proba\big(J(\shift h^{-1}(z))\big) =\proba\big(\shift
J(h^{-1}(z))\big)\, .
\end{equation}

\begin{theorem}\label{existe} Let $p$ be a family of transition
  probabilities satisfying the non-nullness and the
  continuity conditions \ref{nonnull} and \ref{continuity}.  Then
\begin{enumerate} 
\item The map $T$ defined above can be continously extended to a
  monotone increasing map on each inteval $I_j=]\eta_{j-1},\eta_{
    j}[$, with $j=1,\ldots,\cardin$, with end points
    $0=\eta_{0}<\eta_{1}<\ldots<\eta_{\cardin}=1$ defined by \[
    \eta_{k}=h(\cardin_{-\infty}^{-1}k)=h(1_{-\infty}^{-1}(k+1))\,
    \mbox{, for}\, k=1,\dots, \cardin-1\,.  \]
\item The extended map (also denoted by $T$) is a topological Markov
  map and the Lebesgue measure is invariant by $T$ and
  ergodic. Moreover the regular versions of the conditional
  probabilities associated to the sequence of dynamical partitions are
  given by $p$.
\item The map $T$ is differentiable outside $h(\ambigus)$ and for each
  $\omega\in \Omega\backslash h(\ambigus)$ we have
\[
T'(\omega)=\frac{1}{p(h^{-1}(\omega)_0\, |\, h^{-1}(\omega)_{-\infty}^{-1})}\, .
\]
In this formula we denote the successive elements of the sequence
$h^{-1}(\omega) \in A_{-\infty}^0$ by $h^{-1}(\omega)_{-\infty}^{0}$.
\item For  $\omega \in h(\ambigus)$, wih
\[
\omega=h\big(\cardin_{-\infty}^{-(k+2)}az_{-k}^0\big)
=h\big(1_{-\infty}^{-(k+2)}(a+1)z_{-k}^0)\big)\, ,
\] 
for some $
a  \in \{1,\ldots, \cardin-1\}
$ and some integer $k\ge -1$, then the left and right derivatives of
$T$ at $\omega$ exist and are given by
\[
\frac{1}{p\big(z_0\,\big|\, z_{-k}^{-1}a\cardin_{-\infty}^{-(k+2)}\big)} \quad
\mbox{and} \quad \frac{1}{p\big(z_0\,\big|\,
  z_{-k}^{-1}(a+1)1_{-\infty}^{-(k+2)}\big)}\,,
\]
respectively.
\item In particular, if $p$ is such that for any $a \in \{1,\ldots,
  \cardin-1\}$ and any integer $k\ge 0$ and any $z^0_{-k}$, we have
\begin{equation}\label{equal}
p\big(z_0\,\big|\, z_{-k}^{-1}a\cardin_{-\infty}^{-(k+2)}\big)=
p\big(z_0\,\big|\, z_{-k}^{-1}(a+1)1_{-\infty}^{-(k+2)}\big)\,,
\end{equation}
then the map $T$ is piecewise $C^1$.
\item If the continuity rate $\beta_k$, defined in \ref{betak}, decays
  exponentially fast, and conditions \eqref{equal} are satisfied, then
  the map $T$ is piecewise $C^{1+\alpha}$, where $\alpha >0$ depends
  on the exponential rate of decay of $\beta_k$.
\end{enumerate}
 \end{theorem} 

The proof of Theorem \ref{existe} will use several times the following lemma

\begin{lemma}\label{integrale}
For any pair of points $u <v$ in $\Omega\backslash h(\ambigus)$ and
belonging to the same monotonicity interval $I_j=]\eta_{j-1},\eta_{
    j}[$ of $T$, for any $j=1,\ldots,\cardin-1$, we have
\[
T(v)-T(u)=\int_u^v\frac{\lambda(d\omega)}{p(h^{-1}(\omega)_0\, 
|\, h^{-1}(\omega)_{-\infty}^{-1})}\, .
\]
\end{lemma}

\begin{proof}
By definition
\[
\int_u^v\frac{d\lambda(\omega)}{p(h^{-1}(\omega)_0\, |\,
  h^{-1}(\omega)_{-\infty}^{-1})}=
\int_{h^{-1}([u,v])}\frac{d\P(z^0_{-\infty})}{p(z_0\, |\,
    z^{-1}_{-\infty})}
\]

\begin{equation}\label{integre2}
= \int \frac{\gun_{J(h^{-1}(v))\backslash J(h^{-1}(u))}(z^0_{-\infty})}
{p(z_0 |\, z^{-1}_{-\infty})} d\P(z^0_{-\infty})\, .
\end{equation}
In the above formula, $\gun_{J(h^{-1}(v))\backslash J(h^{-1}(u))}$
denotes the characteristic function of the set $J(h^{-1}(v))\backslash
J(h^{-1}(u))$. Since $u$ and $v$ by hypothesis belong to the same
monotonicity interval, we have that
\[
h^{-1}(u)_0=h^{-1}(v)_0\, .
\]

Let $f:A_{-\infty}^{0}\rightarrow \R$ be the function 
\[
f(z_{-\infty}^0)=\frac{\gun_{J(h^{-1}(v))\backslash J(h^{-1}(u))}(z^0_{-\infty})}
{p(z_0 |\, z^{-1}_{-\infty})}=
\frac{\gun_{\{z_0=h^{-1}(v)_0\}}\;\gun_{\shift(J(h^{-1}(v))\backslash J(h^{-1}(u)))}(z^{-1}_{-\infty})}
{p(z_0 |\, z^{-1}_{-\infty})}\, .
\]
Using the invariance of $\P$ (see \eqref{invariant}) with the
function $f$, we can rewrite the integral \eqref{integre2} as
\[
\int \frac{\gun_{J(h^{-1}(v))\backslash J(h^{-1}(u))}(z^0_{-\infty})}
{p(z_0 |\, z^{-1}_{-\infty})} d\P(z^0_{-\infty})=
\int\gun_{\shift(J(h^{-1}(v))\backslash
  J(h^{-1}(u)))}(z^{-1}_{-\infty}) d\P(z^{-1}_{-\infty})\, .
\]
Now we observe that 
\[
\shift(J(h^{-1}(u))\subset \shift(J(h^{-1}(v))
\]
and therefore
\[
\int\gun_{\shift(J(h^{-1}(v))\backslash
  J(h^{-1}(u)))}(z^{-1}_{-\infty})d\P(z^{-1}_{-\infty})=
\P\left\{\shift(J(h^{-1}(v))\right\}-\P\left\{\shift(J(h^{-1}(u))\right\}\, .
\]
Now it is enough to use equality \eqref{Tegal} to conclude the proof.
\end{proof}

We can now prove Theorem \ref{existe}.
\begin{proof} 

Assertion 1 of the theorem follows directly from Lemma \ref{Tegal}.

For Assertion 2, we start by observing  that for
  $i=1,\ldots,\cardin$ we   have
\[
\lim_{\omega\nearrow \eta_i}T(\omega)=1
\]
and for $i=0,\ldots,\cardin-1$
\[
\lim_{\omega\searrow \eta_i}T(\omega)=0\;.
\]

The topological Markov property follows from the piecewise
monotonicity of $T$.

The invariance and ergodicity of the Lebesgue measure $\lambda$ follows
from the fact that $T$ and the shift $\shift$ are conjugated by
$h$. 

To prove that $p$ is the regular version of the conditional
probability we start with equality
$$
\lambda\big(I_{x_{-k}^{0}}\big)=\proba\big(C(x_{-k}^{0})\big)
$$
where
$$
I_{x_{-k}^{0}}=\big\{\omega\,\big|\, T^{j}(\omega)\in I_{x_{-j}}\,, j=0\ldots k\big\}
\;.
$$
Therefore, for any $x_{-\infty}^{0}\in A_{-\infty}^{0}$
$$
\lim_{k\to\infty}\frac{\lambda(I_{x_{-k}^{0}})}{\lambda(I_{x_{-k}^{-1}})}
= \lim_{k\to\infty}\frac{\lambda(I_{x_{-k}^{0}})}{\lambda(T(I_{x_{-k}^{0}}))}
=\lim_{k\to\infty}\frac{\proba\big(C(x_{-k}^{0})\big)}
{\proba\big(C(x_{-k}^{-1})\big)}=p(x_{0}\,|\,x_{-\infty}^{-1})\;,
$$
where the last equality follows from the continuity of the family of
transition probabilities $p$. 
 
Assertions 3, 4 and 5 follow directly from Lemma \ref{integrale}, and the
finiteness of the derivative follows from the non-nullness assumption.

To prove Assertion 6, we first observe that
the exponential decay of the continuity rate  $\beta_{k}$ implies that
there
exists two constants $C>0$ and $0<\rho<1$ such that for any $k\ge 1$
\begin{equation}\label{crho}
\beta_{k}\le C\; \rho^{k}\;.
\end{equation} 
Let
$$
\gamma=
\frac{1}{\sup_{x_{-\infty}^{0}\in A_{-\infty}^{0} }p(x_{0}\,|\,x_{-\infty}^{-1})}\;,
$$
and
$$
\Gamma=
\frac{1}{\inf_{x_{-\infty}^{0}\in A_{-\infty}^{0} }p(x_{0}\,|\,x_{-\infty}^{-1})}\;.
$$

From the non-nullness assumption it follows immediately that
$\gamma>1$, and $\Gamma<\infty$. 
For  $\omega$ and $\omega'$ in the same interval of monotonicity 
$I_{j}$, let 
$$
m=\delta\big(h^{-1}(\omega),h^{-1}(\omega')\big)\;,
$$
where $\delta$ was defined in \ref{dist}. 
Let
$$
M=\left[-\frac{\log|\omega-\omega'|}
{\log\gamma}\right]\;,
$$
where $[\;]$ denotes the integer part.

We first consider the case 
$
m >M
$
Then from \eqref{crho} we have
$$
\big|p(h^{-1}(\omega)_{0}|h^{-1}(\omega)_{-\infty}^{-1})
-p(h^{-1}(\omega')_{0}|h^{-1}(\omega')_{-\infty}^{-1})\big|
\le C\;\rho^{\delta\big(h^{-1}(\omega),h^{-1}(\omega')\big)}
$$
$$
\le C\;\rho^{-1}\;\rho^{-\log|\omega-\omega'|/\log\gamma}
=C\;\rho^{-1}\;|\omega-\omega'|^{-\log\rho/\log\gamma}\;.
$$
This implies that
$$
\big|T'(\omega)-T'(\omega')\big|=\left|
\frac{1}{p(h^{-1}(\omega)_{0}|h^{-1}(\omega)_{-\infty}^{-1})}
-\frac{1}{p(h^{-1}(\omega')_{0}|h^{-1}(\omega')_{-\infty}^{-1})}
\right|
$$
$$
\le \Gamma^{2}\,C\;\rho^{-1}\;|\omega-\omega'|^{-\log\rho/\log\gamma}\;.
$$
We now consider the case $m\le M$. 
If
$$
|\omega-\omega'|^{1/2}\,\Gamma^{m}>
\min\big\{\lambda(I_{1}),\lambda(I_{\cardin})\big\}\;, 
$$
we have
$$
m\ge -\frac{1}{2\log\Gamma}\log|\omega-\omega'|
+\frac{\log \min\big\{\lambda(I_{1}),\lambda(I_{\cardin})\big\}}{\log\Gamma}\;.
$$
The same estimate as before implies
$$
\big|T'(\omega)-T'(\omega')\big|\le
\Gamma^{2}\,C\;\rho^{-1}\;
\rho^{\log \min\big\{\lambda(I_{1}),\lambda(I_{\cardin})\big\}/\log\Gamma}
|\omega-\omega'|^{-\log\rho/(2\log\gamma)}\;.
$$
Finally if 
$$
|\omega-\omega'|^{1/2}\,\Gamma^{m}\le 
\min\big\{\lambda(I_{1}),\lambda(I_{\cardin})\big\}\;, 
$$
we have, assuming $\omega'>\omega$, that
$$
h^{-1}(\omega)=h^{-1}(\omega)_{-\infty}^{-m-2-M/2}\cardin_{-m-1-M/2}^{-m-1}h^{-1}(\omega)_{-m}
h^{-1}(\omega)_{-m+1}^{0}
$$
and 
$$
h^{-1}(\omega')=h^{-1}(\omega')_{-\infty}^{-m-2-M/2}1_{-m-1-M/2}^{-m-1}
(h^{-1}(\omega)_{-m}+1)
h^{-1}(\omega)_{-m+1}^{0}\;.
$$
From inequality \eqref{crho} we get
$$
\big|p(h^{-1}(\omega)_{0}\,|\,h^{-1}(\omega)_{-m+1}^{-1}h^{-1}(\omega)_{-m}
\cardin_{-m-1-M/2}^{-m-1}h^{-1}(\omega)_{-\infty}^{-m-2-M/2})
$$
$$
-p(h^{-1}(\omega)_{0}\,|\,h^{-1}(\omega)_{-m+1}^{-1}h^{-1}(\omega)_{-m}
\cardin_{-\infty}^{-m-1}\big|\le C\,\rho^{m+M/2}
$$
and
$$
\big|p(h^{-1}(\omega')_{0}\,|\,h^{-1}(\omega')_{-m+1}^{-1}h^{-1}(\omega')_{-m}
1_{-m-1-M/2}^{-m}h^{-1}(\omega')_{-\infty}^{-m-2-M/2})
$$
$$
-p(h^{-1}(\omega')_{0}|h^{-1}(\omega')_{-m+1}^{-1}h^{-1}(\omega')_{-m}
1_{-\infty}^{-m-1}\big|\le C\,\rho^{m+M/2}\;.
$$
Observing that 
$$
h^{-1}(\omega)_{-m+1}^{0}=h^{-1}(\omega')_{-m+1}^{0}\;,
$$
$$
 h^{-1}(\omega')_{-m}=h^{-1}(\omega)_{-m}+1\;,
$$
and using Assumption \ref{equal}, we obtain
$$
\big|p(h^{-1}(\omega)_{0}\,|\,h^{-1}(\omega)_{-m+1}^{-1}h^{-1}(\omega)_{-m}
\cardin_{-m-1-M/2}^{-m-1}h^{-1}(\omega)_{-\infty}^{-m-2-M/2})
$$
$$
-p(h^{-1}(\omega')_{0}\,|\,h^{-1}(\omega')_{-m+1}^{-1}h^{-1}(\omega')_{-m}
1_{-m-1-M/2}^{-m}h^{-1}(\omega')_{-\infty}^{-m-2-M/2})\big|
$$
$$
\le  2\,C\,\rho^{m+M/2}\;.
$$
The conclusion follows as in the two other cases.

\end{proof}
\section{The case of chains with memory of variable length.}\label{vl}

Stochastic chains with memory of variable length appeared in the
pionering paper by Rissanen (1983) \nocite{rissanen} as a universal
system for data compression. We briefly recall the definition of this class of stochastic chains . 

Given a finite alphabet $A$, we define the basic notion of {\sl context tree}.
\begin{definition}
A set of strings 
\[
\tau \subset \bigcup_{k \ge 1}A_{-k}^{-1}\,  \bigcup\,  A_{-\infty}^{-1}
\]
is a context tree if 
\begin{enumerate}
\item $\bigcup_{w\in \tau}C(w)=A_{-\infty}^{-1};$
\item for any pair $w$ and $w'$ of  elements of $\tau$, if  $w\neq w'$, then $C(w) \cap C(w') =\emptyset$.
\end{enumerate}
\end{definition}
In the above definition $w$ and $w'$ denote two sequences, either finite or infinite, and $C(w)$ is the set of all elements of $A_{-\infty}^{-1}$ having the string $w$ as a suffix, {\sl i. e.} having $w$ as final sequence. In case $w$ is finite, $C(w)$ is a cylinder. In case $w$ is infinite $C(w)$ is the unitary set whose unique element is $w$. The name {\sl context tree} comes from the fact that $\tau$ can be described by the leaves of a rooted tree. The strings belonging to $\tau$ are called {\sl contexts}.

\begin{definition}
A probabilistic context tree is a pair $(\tau, p)$, where $\tau$ is a context tree and
\[
p=\{p(\cdot)\, |\, w) \, |\, w \in \tau\}
\]
is a family indexed by $\tau$ of probability measures on the set $A$.
\end{definition}

Given a probabilistic context tree $(\tau, p)$, we define a family of infinite order transition probabilities $\tilde{p}$ on $A$ as follows. For any sequence $x_{-\infty}^{-1} \in A_{-\infty}^{-1} $, and for any symbol $a \in A$
\begin{equation}\label{tree_model}
\tilde{p}(a\,|\, x_{-\infty}^{-1})=p(a\,|\, w)\,
\end{equation}
where $w$ is the unique element of $\tau$, such that $x_{-\infty}^{-1} \in C(w)$.

\begin{definition}
A stochastic chain of infinite order is said to have a memory of variable length described by a probabilistic context tree $(\tau, p)$ if its family of transition probabilities satisfies conditions \eqref{tree_model}.
\end{definition}

Intuitively speaking  in a chain with memory of variable length, at each time step, to predict the next symbol, it is enough to use the past steps corresponding to the context associated to this past.

The question we address in this section is to characterize the maps associated to transition probabilities defined by a probabilistic context tree.  This is the content of the following theorem. 
\begin{theorem}
Let $T$ be a topological Markov expanding map of the interval with alphabet of monotonicity intervals $A$, and with the Lebesgue measure invariant and ergodic. Assume there is a tree of contexts $\tau$ on the alphabet $A$ such that
$$
\sum_{k=1}^{\infty}\sum_{x_{-k}^{-1}\in\tau\,\cap\, A^{-1}_{-k}} \lambda\big(C(x^{-1}_{-k})\big)=1\;,
$$
and for any $x_{-k}^{-1}\in \tau$, for any $a\in A$ and for any $\omega$ and $\omega'$ satisfying
$$
W(\omega)_0^k=a\,x_{-1}\,\ldots\,x_{-k}\,,\qquad\mathrm{and}
\qquad W(\omega')_0^k=a\,x_{-1}\,\ldots\,x_{-k}\;,
$$
we have
$$
T'(\omega)=T'(\omega')\;.
$$
Then the family of transition probabilites associated to the map $T$ by theorem \ref{maptochain}  is a chain with variable length whose contexts are almost surely finite.

Conversely, given a family of transition probabilities which is a chain of variable length with almost surely finite contexts (for an invariant measure), then the associated map by \eqref{Tegal} (see also Theorem \ref{existe}) is piecewise affine with derivatives satisfying the above property. 
\end{theorem}
\begin{proof}
The result follows directly from Theorems \ref{maptochain} and \ref{existe}.
\end{proof}

\section*{Acknowledgements}
This work is part of USP project MaCLinC, ``Mathematics, computation, 
language and the brain", USP/COFECUB project
``Stochastic systems with interactions of variable range'' and CNPq project
476501/2009-1. AG is partially supported by a CNPq fellowship (grant
305447/2008-4).  P.C. thanks Numec-USP for its kind hospitality.

\bibliographystyle{abbrv}
\bibliography{biblio-07-06-2012.bib}
\end{document}